\documentclass [12pt,a4paper]{amsart}
\usepackage{}
\usepackage{stmaryrd}
\usepackage{textcomp}
\usepackage{amsmath}
\usepackage{cases}
\usepackage{amscd}
\usepackage{epsfig}
\usepackage{graphicx}
\usepackage{verbatim}
\usepackage{amssymb}
\usepackage{amsfonts}
\usepackage{amsbsy}
\usepackage{amsthm}
\usepackage{amstext}
\usepackage{amsopn}
\usepackage[dvips]{color}
\usepackage{indentfirst}
\usepackage{mathrsfs}
\usepackage{enumerate}

\newtheorem{thm}{Theorem}[section]

\newtheorem{cor}[thm]{Corollary}
\newtheorem{lem}[thm]{Lemma}

\newtheorem{rem}[thm]{Remark}

\newtheorem{eg}[thm]{Example}

\newcommand{\PU}{\textrm{PU}}
\newcommand{\HC}{{\textbf{H}_{\mathbb{C}}^2}}

\renewcommand{\Im}{\,{\operatorname{Im}\,}}
\renewcommand{\Re}{\,{\operatorname{Re}\,}}
\newcommand{\Z}{\mathbb{Z}}

\begin{document}
\title{ Notes on complex hyperbolic triangle groups of type $(m,n,\infty)$}
\author{Li-jie Sun}
\date{\today}
\maketitle
\begin{abstract}
In this paper we mainly pay attention to the complex hyperbolic triangle groups of type $(m,n,\infty)$ and discuss the discreteness. From the results more explicit conclusions about the triangle groups of type $(n,\infty,\infty)$ will also be given.
\end{abstract}

\section{ Introduction}
\par A complex hyperbolic triangle is a triple $(C_{1}, C_{2}, C_{3})$ of complex geodesics in $\HC$. If the complex geodesics $C_{k-1}$ and $C_{k}$ meet at the angle $\frac{\pi}{p}, \frac{\pi}{q}, \frac{\pi}{r}$ ($p, q, r \in \mathbb{Z}$), where the indices are taken mod 3, we call the triangle $(C_{1}, C_{2}, C_{3})$ a $(p, q, r)-triangle$. We call $\Gamma$ a $(p, q, r)-$\emph{triangle group}, if $\Gamma$ is generated by three complex reflections $I_{1}, I_{2}, I_{3}$ in the sides $C_{1},C_{2},C_{3}$ of a $(p,q,r)$-triangle. Throughout this paper we will use $C_{j}, I_{j}, \Gamma$ to denote the
complex geodesic, complex reflection, and the complex hyperbolic triangle group respectively, unless otherwise stated.
\par The deformation theory of complex hyperbolic triangle groups was begun in \cite{GP1}. Goldman and Parker investigated $\Gamma$ of type $(\infty,\infty,\infty)$ (complex
hyperbolic ideal triangle group) and gave the necessary and sufficient conditions for ideal triangle group $\Gamma$  to be discrete embedded. Especially the necessary condition for $\Gamma$ of type $(\infty,\infty,\infty)$ to be discretely embedded in $\PU(2,1)$ is that the product of the three generators $I_{1}I_{2}I_{3}$ is not elliptic. They conjectured the necessary condition is also sufficient. Since then there have been many findings. Schwartz in \cite{Sch1} proved this conjecture and also verified that such a group is non-discrete if $I_{1}I_{2}I_{3}$ is elliptic. Recently Parker, Wang and Xie in \cite{PWX} show that the group of type $(3,3,n)$ is discrete if and only if $I_{1}I_{3}I_{2}I_{3}$ is non-elliptic which is a conjecture in \cite{Sch2}. Parker \cite{Par3} explored groups of type $(n,n,n)$ such that $I_{1}I_{2}I_{3}$ is regular elliptic. In this case there are some discrete groups. In the same fashion as the proof due to Schwartz, Wyss-Gallifent proved the Schwartz's statement for $\Gamma$ of type $(n,n,\infty)$ in $\text{\cite[Lemma 3.4.0.19]{Wy}}$. Pratoussevitch improved the result of Wyss-Gallifent in \cite{Pra2}. Also Kamiya, Parker and Thompson identified the non-discrete classes of $(n,n,\infty)-$triangle groups using the result, see \cite{KPT2}. It is interesting to think about whether the same statement holds for other type complex hyperbolic triangle groups, such as type $(m,n,\infty)(m\neq n).$ In this paper, we give the affirmative answer by the similar argument.

\par This paper is arranged as follows. Section 2 mainly consists of background about complex hyperbolic space and its holomorphic isometry group. Then we give three necessary conditions for $(m,n,\infty)$ $\text{-}$ triangle groups to be discrete in Section 3. After that we consider non-discrete cases of $(n,\infty,\infty)$ $\text{-}$ triangle groups in Section 4.

\section{Preliminaries}
\par Let $\mathbb{C}^{2,1}$ denote the vector space $\mathbb{C}^{3}$ equipped with the Hermitian form
$$\langle z,w \rangle=z_{1}\overline{w}_{1}+z_{2}\overline{w}_{2}-z_{3}\overline{w}_{3}$$
of signature (2,1). We denote by $\mathbb{CP}^{2}$ the complex projectivisation of $\mathbb{C}^{2,1}$ and by $\mathbb{P}:$ $\mathbb{C}^{2,1}\setminus \{0\}\rightarrow\mathbb{CP}^{2}$ a natural projectivisation map.  We call a vector $z\in \mathbb{C}^{2,1}$ \textit{negative, null, or positive}, according as $\langle z,z \rangle$ is negative, zero, or positive respectively. The \emph{complex hyperbolic 2-space} $\HC$ is defined as the complex projectivisition of the set of negative vectors in $\mathbb{C}^{2,1}$. It is called the standard projective model of complex hyperbolic space. Its boundary $\partial \HC$ is defined as the complex projectivisation of the set of null vectors in $\mathbb{C}^{2,1}$. This will also form the \emph{unit ball model} whose boundary is the sphere $\mathbb{S}^3$.
\par The complex hyperbolic plane $\HC$ is a $K\ddot{a}hler$ manifold of constant holomorphic sectional curvature -1. The holomorphic isometry group of $\HC$ is the projectivisation $\textrm{PU}(2,1)$ of the group $\textrm{U}(2,1)$ of complex linear transformation on $\mathbb{C}^{2,1}$, which preserves the Hermitian form.
\par Let $x,y\in \HC$ be points corresponding to vectors $\tilde{x},\tilde{y}\in \mathbb{C}^{2,1}$. Then the \textit{Bergman metric} $\rho$ on $\HC$ is given by
$$\cosh^2\Big(\frac{\rho(x,y)}{2}\Big)=\frac{\langle \tilde{x},\tilde{y}\rangle\langle \tilde{y},\tilde{x}\rangle}{\langle \tilde{x},\tilde{x}\rangle\langle \tilde{y},\tilde{y}\rangle}.$$

\par It will be convenient for us to choose a  particular model of the complex hyperbolic space which is adapted for our requirements; namely, one with the distinguished point $q_{\infty}$ on the boundary and a set of coordinates with respect to this point. This set-up is generalised by the the Siegel domain model $\mathfrak{S}$ of $\HC$ with horospherical coordinates, see \cite{GP2}. In these coordinates $z\in\mathfrak{S}$ is given by $z=(\xi,v,u)\in\mathbb{C}\times\mathbb{R}\times\mathbb{R}_{+}.$ Similarly, points in $\partial\HC=\mathbb{C}\times\mathbb{R}\cup\{q_\infty\}$ are either $z=(\xi,v,0)\in\mathbb{C}\times\mathbb{R}\times\{0\}$ or the point at infinity $q_{\infty}.$
There is unique complex projective hyperplane $\textrm{H}_{\infty}\subset\mathbb{CP}^{2}$ that is tangent to $\partial\HC$ at $q_{\infty}.$ Using affine coordinates on $\mathbb{CP}^{2}-\textrm{H}_{\infty}$ complex hyperbolic space is realised as a \emph{Siegel domain.}
\par The 3-dimensional \emph{Heisenberg group} $\mathfrak{N}$ is the set $\mathbb{C}\times\mathbb{R}$ with the group law
$$(\xi_{1},v_{1})\lozenge(\xi_{2},v_{2})=(\xi_{1}+\xi_{2},v_{1}+v_{2}+2\Im(\xi_{1}\overline{\xi_{2}})).$$
The inverse of $(\xi_{1},v_{1})$ is
$$(\xi_{1},v_{1})^{-1}=(-\xi_{1},-v_{1}).$$
\par The boundary of the half-space model of real hyperbolic geometry is identified with the one-point compactification of Euclidean space. In the same way, the boundary of the Siegel domain may be identified with the one-point compactification of the Heisenberg group. In order to see how $\mathfrak{S}$ relates to the standard projective model of $\textbf{H}_{\mathbb{C}}^{2}$ we define the map $\psi:$ $\overline{\mathfrak{S}}\longrightarrow \mathbb{CP}^2$ by

\begin{displaymath}
\psi: (\xi, v, u)\longmapsto
\left[ \begin{array}{ccc}
\xi\\
\frac{1}{2}(1-|\xi|^2-u+i v)\\
\frac{1}{2}(1+|\xi|^2+u-i v)
\end{array} \right]\quad \textmd{for}\quad(\xi, v, u)\in  \overline{\mathfrak{S}}-\{q_{\infty}\},
\end{displaymath}

and $\psi(q_{\infty})=[0,-1,1]^{t}.$
\par The Heisenberg norm is given by
$$|(\xi,v)|=\big||\xi|^2-iv\big|^{\frac{1}{2}}.$$
This gives rise to a metric, the \textit{Cygan metric} $\rho_{0}$ on the Heisenberg group $\mathfrak{N}$ by
$$\rho_{0}((\xi_{1},v_{1}),(\xi_{2},v_{2}))=\big|(\xi_{1},v_{1})^{-1}\lozenge(\xi_{2},v_{2})\big|=\big||\xi_{1}-\xi_{2}|^2-iv_{1}+iv_{2}-2i\Im(\xi_{1}\overline{\xi_{2}})\big|^{\frac{1}{2}}.$$
We can extend the Cygan metric to $\overline{\HC}-\{q_{\infty}\}$ as follows (\cite{Par2})
\[
\rho_{0}((\xi_{1},v_{1},u_{1}),(\xi_{2},v_{2},u_{2}))=\big||\xi_{1}-\xi_{2}|^2+|u_{1}-u_{2}|-iv_{1}+iv_{2}-2i\Im(\xi_{1}\overline{\xi_{2}})\big|^{\frac{1}{2}}.
\]

\par  A \emph{complex geodesic} is a complex projectivisation of a two dimensional complex subspace of $\mathbb{C}^{2,1}$. Given two points $x$ and $y$ in $\HC\cup \partial \HC$, lifting $x$ and $y$ to $\tilde{x}$ and $\tilde{y}$ in $\mathbb{C}^{2,1}$ respectively, and then taking $\widetilde{C}$ to be the complex span of $\tilde{x}$ and $\tilde{y}$. We define the complex geodesic $C$ to be the projectivisation of $\widetilde{C}$, which can be uniquely determined by a positive vector $p\in\mathbb{C}^{2,1},$ i.e. $C=\pi(\{z\in\mathbb{C}^{2,1}|\langle z, p\rangle=0\})$. We call $p$ a \emph{polar vector} to $C.$

\par Recall that a chain is the intersection of a complex geodesic with $\partial{\textbf{H}}_{\mathbb{C}}^2$. For $z$ $\in$ $\mathbb{C}$, the \textit{z-chain} is the chain having polar vector $(1,-\overline{z},\overline{z})^{t}.$ The $z$-chain is the vertical chain in $\mathfrak{N}$ through the point $(z,0)$. For $z,r$ $\in$ $\mathbb{R}$, the $(z, r)$-$chain$ is the chain having polar vector $(0,1+r^2+iz,1-r^2-iz)^{t}.$ The $(z, r)$-chain is the circle with radius $r$ centered at the origin in $\mathbb{C}\times\{z\}\subset \mathfrak{N}$. One can see more details in \S4.3 of \cite{Gol}. It is straightforward to show that the only chains through $\infty$ are vertical. Other chains are various ellipses (perhaps circles) which project to circle via $\mathbb{C}\times\mathbb{R}\rightarrow\mathbb{C}$. Specifically, the unit circle in $\mathbb{C}\times\{0\}$ and vertical lines (with the infinite point) are all chains.

\par The involution (complex reflection of order 2) in $C$ is represented by an element $I_{C}\in \textrm{SU}(2,1)$ that is given by
$$I_{C}=-z+2\frac{\langle z, p\rangle}{\langle p, p\rangle}p,$$ where $p$ is a polar vector of $C$. There is a one-to-one correspondence between complex geodesics and chains, therefore we also say $I_{C}$ is the involution on $
\partial{C}$.

\par An automorphism $g$ of $\HC$ lifts to a unitary transformation $\tilde{g}$ of $\mathbb{C}^{2,1}$ and the fixed points of $g$ on $\mathbb{P}(\mathbb{C}^{2,1})$ correspond to eigenvalues of $\tilde{g}$. An automorphism $g$ is $elliptic$ if it fixes at least one point in $\HC$, $parabolic$ if it has a unique fixed point on $\partial \HC$, and $loxodromic$ if it fixes a unique pair of points on $\partial \HC$. An elliptic element $g$ is called $regular$ $elliptic$ if its eigenvalues are pairwise distinct. Otherwise we call it \textit{boundary elliptic}, in which case the element has a multiple eigenvalue with a two dimensional eigenspace.

\par Define the discriminant polynomial
\begin{equation}\label{dis}
f(z)=|z|^4-8\Re(z^3)+18|z|^2-27.
\end{equation}
From $\text{\cite[Theorem 6.2.4]{Gol}}$, we know an element $g\in \textrm{SU(2,1)}$ is regular elliptic if and only if $f(\tau(g))<0,$ where $\tau(g)$ is the trace of $g$.

\par See \cite{Gol} for more details about complex hyperbolic space.

\section{Complex hyperbolic triangle group of type $(m,n,\infty)$}
\par By conjugation in $\textrm{PU}(2,1)$, we can take three involutions $I_{j}$ in
$C_{j}$ such that $\partial C_{1},~\partial C_{2},~\partial C_{3}$ are (0,1)-chain, ${z_{1}}$-chain, ${z_{2}}$-chain respectively, where
$z_{1}= \cos(\pi/n)$, $z_{2}= e^{i\theta}\cos(\pi/m)$ according to $\text{\cite[Lemma 3.1.0.7]{Wy}}$. Then the three
polar vectors correspondingly are
\begin{displaymath}
{p_{1}} =
\left[ \begin{array}{ccc}
0\\
1 \\
0
\end{array} \right],\quad
{p_{2}} =
\left[ \begin{array}{ccc}
1 \\
-z_{1} \\
z_{1}
\end{array} \right],\quad
{p_{3}} =
\left[ \begin{array}{ccc}
1\\
-\overline z_{2}\\
\overline z_{2}
\end{array} \right].
\end{displaymath}

It is easy to obtain the three vertices

\begin{displaymath}
{\tilde{u}_{1}} =
\left[ \begin{array}{ccc}
0\\
1 \\
-1
\end{array} \right],\quad
{\tilde{u}_{2}} =
\left[ \begin{array}{ccc}
z_{2} \\
0 \\
1
\end{array} \right],\quad
{\tilde{u}_{3}} =
\left[ \begin{array}{ccc}
\overline z_{1}\\
0\\
1
\end{array} \right].
\end{displaymath}
The involutions on the complex chains $\partial C_{1}, \partial C_{2}, \partial C_{3}$ are respectively as follows
\begin{displaymath}
{I_{1}} =
\left[ \begin{array}{ccc}
-1&\quad 0 \quad& 0 \\
0 & 1 & 0 \\
0 & 0 & -1
\end{array} \right],\quad
{I_2} =
\left[ \begin{array}{ccc}
1 &\quad -2s_{1} \quad& -2s_{1} \\
-2s_{1} & 2s_{1}^2-1 & 2s_{1}^2 \\
2s_{1} & -2s_{1}^2 & -2s_{1}^2-1
\end{array} \right],
\end{displaymath}
\begin{displaymath}
{I_3} =
\left[ \begin{array}{ccc}
1 &\quad -2s_{2}e^{i\theta} \quad& -2s_{2}e^{i\theta} \\
-2s_{2}e^{-i\theta} & 2s_{2}^2-1 & 2s_{2}^2 \\
2s_{2}e^{-i\theta} & -2s_{2}^2 & -2s_{2}^2-1
\end{array} \right],
\end{displaymath}
where $s_{1}= \cos(\pi/n)$, $s_{2}= \cos(\pi/m)$.
Define the parameter of the $(m,n,\infty)-$triangle \textit{angular invariant} $\alpha$ by
$$\alpha=\arg\left(\prod\limits_{i=1}^3\langle p_{i-1}, p_{i+1}\rangle\right)=\arg(z_{1}z_{2})=\theta.$$
\begin{rem}
1. For complex hyperbolic triangle group $\Gamma$ of type $(n,n,\infty)$, we can take the following three polar vectors
\begin{displaymath}
{p_{1}} =
\left[ \begin{array}{ccc}
0\\
1 \\
0
\end{array} \right],\quad
{p_{2}} =
\left[ \begin{array}{ccc}
1 \\
-\overline z \\
\overline z
\end{array} \right],\quad
{p_{3}} =
\left[ \begin{array}{ccc}
1\\
z\\
-z
\end{array} \right],
\end{displaymath}
where $z=\cos (\frac {\pi}{n})e^{i\theta_{0}}.$ These three normalised polar vectors were also used in \cite{Kam3}, \cite{KPT2}.
\\2. One can compare our parameter $\alpha$ of the space of complex hyperbolic triangles with another parameter $t$ by Wyss-Gallifent \cite{Wy} and shall obtain
\begin{equation}\label{par}
\cos\alpha=\frac{t^2-1}{t^2+1}.
\end{equation}

\end{rem}
\par Now we give the affirmative answer about the Schwartz's statement $\text{\cite[Section 3.3]{Sch1}}$ for the $(m,n,\infty)-$ triangle groups.
\begin{thm}\label{thm3.1}
The complex triangle group $\Gamma$ of type $(m,n,\infty)$ is not discrete if $I_{1}I_{2}I_{3}$ is regular elliptic.
\end{thm}
\begin{proof} If the element $I_{1}I_{2}I_{3}$ is of infinite order, then the cyclic group generated by this element is certainly not discrete. Hence it suffices for us to prove that $I_{1}I_{2}I_{3}$ can not be regular elliptic of finite order. We will only consider the result when $m\neq n$ using the similar method given by Pratoussevitch (see \cite{Pra2}), in which the author proved the case for $m= n$.
\par Assume that the element $I_{1}I_{2}I_{3}$ is regular elliptic of finite order. Without loss of generality, we can write
\begin{equation}\label{tau.0}
\tau=\textmd{tr}(I_{1}I_{2}I_{3})=\omega_{l}^{k_{1}}+\omega_{l}^{k_{2}}+\omega_{l}^{k_{3}},
\end{equation}
for some integers $k_{1},k_{2},k_{3}$, with $k_{1}+k_{2}+k_{3}=0$. Here $\omega_{l}=e^{2\pi i/l}$ and a positive integer $l$ is taken as small as possible.
\par Let $N$ be the least common multiple of $l$, $2m$ and $2n.$ Let $k$ be relatively prime to $N$, $\sigma_{k}$ be the Galois automorphism of $\mathbb{Q}[\omega_{N}]$ given by $\sigma_{k}(\omega_{N})=\omega_{N}^{k}.$ Obviously $\sigma_{k}(t)=t,$ for $t\in\mathbb{N.}$

\begin{lem}\label{lem:3.2}
$\textmd{Re}(\sigma_{k}(\tau))<-1$.
\end{lem}
\begin{proof}
From the explicit form of three involutions $I_{1},I_{2},I_{3}$, we can rewrite the trace of $I_{1}I_{2}I_{3}$ as $\tau=8s_{1}s_{2}e^{i\alpha}-(4(s_{1}^2+s_{2}^2)+1)$, i.e.
\begin{equation}\label{tau}
|\tau+4(s_{1}^2+s_{2}^2)+1)|=(8s_{1}s_{2})^2.
\end{equation}

\par By considering (\ref{tau.0}) the expression of $\tau=\omega_{l}^{k_{1}}+\omega_{l}^{k_{2}}+\omega_{l}^{k_{3}},$ we know
$$\tau\in\mathbb{Q}[\omega_{l}]\subseteq \mathbb{Q}[\omega_{N}],$$
$$2s_{1}=2\textmd{cos}(\pi/n)=\omega_{2n}+\overline{\omega}_{2n}\in\mathbb{Q}[\omega_{N}].$$
Similarly, $2s_{2}\in\mathbb{Q}[\omega_{N}].$ Let $s_{j}'=\sigma_{k}(s_{j})$ for $j=1,2$ ($s_{j}'$ could be equal to $s_{j}$).
Then the equation (\ref{tau}) implies that
$$(\sigma_{k}(\tau)+4(s_{1}^2+s_{2}'^2)+1)(\sigma_{k}(\overline{\tau})+4(s_{1}'^2+s_{2}'^2)+1)=(8s_{1}'s_{2}')^2.$$

Since $\sigma_{k}$ commutes with complex conjugation, we know $s_{j}'\in\mathbb{R}$ and
$$|\sigma_{k}(\tau)+4(s_{1}'^2+s_{2}'^2)+1|=|8s_{1}'s_{2}'|.$$
It follows that $\sigma_{k}(\tau)$ lies on the circle with center at $-(4(s_{1}'^2+s_{2}'^2)+1)$ and radius $|8s_{1}'s_{2}'|$. It is easy to compute
\begin{align*}
-(4(s_{1}'^2+s_{2}'^2)+1)+|8s_{1}'s_{2}'|&= -4(s_{1}'^2-2|s_{1}'s_{2}'|+s_{2}'^2)-1\\
&= -4(s_{1}'\pm s_{2}')^2-1\\
&< -1.
\end{align*}
The last strict inequality is from $s_{1}'\neq s_{2}'$, because $s_{1}\neq s_{2}$ for $m\neq n$.
Hence $\Re(\sigma_{k}(\tau))<-1.$
\end{proof}

Note that The following lemma is essentially Lemma 2 of \cite{Pra2}. We clarified it here again by taking different values for $k$ from \cite{Par2}.

\begin{lem}\label{lem:3.3}
For $i\in\{1,2,3\}$ let
$d_{i}=\frac{l}{(k_{i},l)},$
where $(k_{i},l)$ is the greatest common divisor of $k_{i}$ and $l$. Then
\begin{equation}\label{phi}
\frac{1}{\varphi(d_{1})}+\frac{1}{\varphi(d_{2})}+\frac{1}{\varphi(d_{3})}>1.
\end{equation}
Here $\varphi$ is the Euler phi function.
\end{lem}
\begin{proof}
Let $S(N)=\{k\in\mathbb{Z}\mid 1\leq k<N$ and $(k,N)=1\}$ (note that one can also have the similar definition for $S(d_{i})$).
It follows from Lemma \ref{lem:3.2} that
\begin{equation}\label{1}
\Re\left(\sum_{
\begin{subarray}{c}
k\in S(N)
\end{subarray}
}\sigma_{k}(\tau)\right)<-\varphi(N).
\end{equation}

By assuming $N=l\cdot l'$, we obtain $d_{i}=\frac{l}{(k_{i},l)}=\frac{N}{(k_{i}l',N)}.$ Note that the root of unity $\omega_{l}^{k_{i}}=\omega_{N}^{l'k_{i}}$ is a primitive $d_{i}$th root of unity and
\[\sum_{
\begin{subarray}{c}
k\in S(d_{i})
\end{subarray}
}\omega_{d}^k \in \{-1,0,1\}.\]
 The map $(\Z/N\Z)^{\times }\rightarrow(\Z/d_{i}\Z)^{\times}$ induced by
$\Z/N\Z\rightarrow\Z/d_{i}\Z$ is surjective and with multiplicity $\varphi(N)/\varphi(d_{i}).$ Therefore we obtain the inequality
\begin{equation}\label{2}
\left|
\sum_{\begin{subarray}{l}
k\in S(N)
\end{subarray}
}
\sigma_{k}(\omega_{N}^{l'k_{i}})\right|
\leq\frac{\varphi(N)}{\varphi(d_{i})},
\end{equation}
for $i\in\{1,2,3\}.$
Combing (\ref{1}) and (\ref{2}) we get
\begin{align*}
\varphi(N) & <\left|
\sum_{\begin{subarray}{l}
k\in S(N)
\end{subarray}
}
\sigma_{k}(\tau)\right|\\
& = \left|
\sum_{\begin{subarray}{l}
k\in S(N)
\end{subarray}
}
\sigma_{k}(\omega_{N}^{l'k_{1}}+\omega_{N}^{l'k_{2}}+\omega_{N}^{l'k_{3}})\right|\\
& \leq\left(\frac{1}{\varphi(d_{1})}+\frac{1}{\varphi(d_{2})}+\frac{1}{\varphi(d_{3})}\right)\cdot\varphi(N).\\
\end{align*}
Then the result can be obtained immediately.
\end{proof}
 Using the previous lemma, we could totally follow the statement due to \cite{Pra2} to show that there do not exist appropriate values for $l, k_{1}, k_{2}$ and $k_{3}$ such that (\ref{phi}) holds, i.e.
  $I_{1}I_{2}I_{3}$ can not be regular elliptic of finite order. So $\Gamma$ of type $(m,n,\infty)$ is not discrete when $I_{1}I_{2}I_{3}$ is regular elliptic.
\end{proof}
\par Applying this theorem, we can work out some conditions on $\cos \theta$ for $\Gamma$ with angular invariant $\theta$ of type $(m,n,\infty)$ to be non-discrete. A simple calculation yields that $$\tau=-5-2 \cos(2\pi/m)-2 \cos(2\pi/n)+8e^{i\theta} \cos(\pi/m) \cos(\pi/n)$$ by seeing (\ref{tau}).
Consequently we can obtain the interval of $a=\cos \theta$ $(-1\leq a\leq1)$ corresponding to the non-discrete $\Gamma$ by using the discriminant function (\ref{dis}). In the remaining content $\theta$ is the angular invariant of complex triangle group $\Gamma$ of type $(m,n,\infty),$ unless otherwise stated.
\par We give an example for $m=8$ showing that $\Gamma$ of type $(8,n,\infty)$ is non-discrete if $a\in(a_{n},$ $b_{n}).$ Note that there are no solutions for $a$ when $n\leq10.$
\begin{center}
\begin{table}[h]
\caption{Approximate values of $a_{n},$ $b_{n}$.}
\begin{tabular}{|c|c|c|}
\hline
n & $a_{n}$ & $b_{n}$\\
\hline
11&0.93067&0.93114\\
12&0.93226&0.93268\\
13&0.93318&0.93377\\
14&0.93386&0.93454\\
15&0.93437&0.93512\\
20&0.93575&0.93654\\
30&0.93662&0.93733\\
40&0.93690&0.93757\\
100&0.93719&0.93780\\
200&0.93723&0.93783\\
\hline
\end{tabular}
\end{table}
\end{center}

\par In the following we will use other different ways to find sufficient conditions on $a$ for $\Gamma$ to be non-discrete. Let $g\in \textrm{PU}(2,1)$ be a parabolic element. Define the \emph{translation length} $t_{g}(z)$ of $g$ at $z\in\mathfrak{N}$ by $t_{g}(z)=\rho_{0}(g(z),z).$ For the following discussion, it is necessary to give the explicit form of Ford isometric spheres. Let $h=(a_{ij})_{1\leq i,j \leq 3}$ be an element of $\textrm{PU}(2,1)$ not fixing $\infty$ (let the null vector $\infty$ represent the point $q_{\infty}$ at $\partial\HC$). The isometry sphere of $h$ is the sphere in the Cygan metric with center at $h^{-1}(\infty)$ and radius
$$r_{h}=\sqrt{\frac{2}{|a_{22}-a_{23}+a_{32}-a_{33}|.}}$$
(see \cite{Par1}). Now let's recall the complex hyperbolic versions of $\textmd{J\o rgensen's}$ inequality and Shimizu's lemma.
\begin{lem}{\cite[Lemma 2.2]{KPT1}}\label{lem:3.3}
Let $\textit{A} \in \textrm{SU}(2,1)$  be a regular elliptic element of order $n\geq7$ that preserves a Lagrangian plane (i.e. tr(\textit{A}) is real). Suppose that $A$ fixes a point z $\in \mathbb{H}^2_{\mathbb{C}}$. Let $B$ be any element of $\textrm{SU}(2,1)$ with $B(z)\neq z$. If
 $$\cosh\Big(\frac{\rho(Bz,z)}{2}\Big)\sin\Big(\frac{\pi}{n}\Big)<\frac{1}{2},$$
then $\langle A, B\rangle$ is not discrete and consequently any group containing $A$ and $B$ is not discrete.
\end{lem}
\begin{lem}{\cite[Theorem 2.1]{Par2}}\label{lem:3.4}
Let $G$ be a discrete subgroup of $\textrm{PU}(2,1)$ that contains the Heisenberg translation $g$ by $(\xi,\nu).$ Let $h$ be any element of $G$ not fixing $\infty$ and with isometric sphere of radius $r_{h}$. Then $$r_{h}^2\leq t_{g}(h^{-1}(\infty))t_{g}(h(\infty))+4|\xi|^2.$$
\end{lem}
\par In the sequel we give two necessary conditions for $(m,n,\infty)-$ triangle groups to be discrete using the previous two lemmas.
\begin{thm}\label{thm3.2}
The complex hyperbolic triangle group $\Gamma$ of type $(m,n,\infty)$ is not discrete if $m,n,\theta$ satisfy one of the two following conditions
\\
\\(1) $7\leq n<\infty$ and
\begin{equation}\label{jor}
\Big|\cos^2\Big(\frac{\pi}{n}\Big)+2\cos^2\Big(\frac{\pi}{m}\Big)-4\cos\Big(\frac{\pi}{n}\Big)\cos\Big(\frac{\pi}{m}\Big)\cos \theta+1\Big|<\frac{1}{2}\sin\Big(\frac{\pi}{n}\Big);
\end{equation}
\\(2)(Suppose that $u=\cos^2(\frac{\pi}{m})+\cos^2(\frac{\pi}{n})-2\cos(\frac{\pi}{m})\cos(\frac{\pi}{n})\cos \theta$,
$v=\cos(\frac{\pi}{m})\cos(\frac{\pi}{n})\sin \theta.$)
\begin{equation}\label{shimizu}
|u-2iv|+4u<\frac{1}{4}.
\end{equation}
\end{thm}
\begin{proof}
(1) Let $A=I_{1}I_{2},$ $B=I_{3}$ and $z=z_{0}$ (fixed point of $I_{12})$, where
\begin{displaymath}
{z_{0}} =
\left[ \begin{array}{ccc}
s_{1}\\
0\\
1
\end{array} \right].
\end{displaymath}
By computation, we know
\begin{displaymath}
{B(z_{0})} =
\left[ \begin{array}{ccc}
s_{1}-2s_{2}e^{i\theta}\\
-2s_{1}s_{2}e^{-i\theta}+2s_{2}^2\\
2s_{1}s_{2}e^{-i\theta}-2s_{2}^2-1
\end{array} \right].
\end{displaymath}

It is easy to see $s_{2}\neq s_{1}e^{-i\theta}$ which means $B$ does not fix $z_{0}$, otherwise $C_{1}, C_{2}$ will coincide.
\par Using $J$\o$rgensen's$ $inequality $ stated previously (Lemma \ref{lem:3.3}), we deduce if
$$\Big|\frac{\langle B(z_{0}), z_{0}\rangle}{\langle z_{0},z_{0}\rangle}\Big| \sin\Big(\frac{\pi}{n}\Big)<\frac{1}{2},$$
then $\Gamma$ is not discrete. Simplifying the inequality above, we will obtain \eqref{jor} stated in the theorem.
\\(2) Let $g=I_{2}I_{3}$ and $h=I_{1}$. We see that $g$ is a Heisenberg translation in the form $\Big(2\big(-e^{i\theta}\cos(\frac{\pi}{m})+\cos(\frac{\pi}{n})\big), 8\sin \theta \cos(\frac{\pi}{m}) \cos(\frac{\pi}{n})\Big)$ fixing $\infty$.
It is clear that $h$ does not fix $\infty$ and has the isometric sphere with radius 1. By computing
\begin{displaymath}
{h^{-1}(\infty)} =
\left[ \begin{array}{ccc}
0\\
-1\\
-1
\end{array} \right],\quad
{h(\infty)} =
\left[ \begin{array}{ccc}
0\\
-1\\
-1
\end{array} \right],
\end{displaymath}
we know
\begin{align*}
&  \quad t_{g}(h(\infty))t_{g}(h^{-1}(\infty))\\
& =\rho_{0}^2(g(h(\infty)),h(\infty)) \\
& =\Big|4\Big(\cos^2\Big(\frac{\pi}{m}\Big)+\cos^2\Big(\frac{\pi}{n}\Big)+2\cos\Big(\frac{\pi}{m}\Big)\cos\Big(\frac{\pi}{n}\Big) \cos \theta\Big)
-8i\sin \theta \cos\Big(\frac{\pi}{m}\Big) \cos\Big(\frac{\pi}{n}\Big)\Big|.
\end{align*}
Then the inequality \eqref{shimizu} can be obtained by applying \emph{Shimizu's lemma for complex hyperbolic space} (Lemma \ref{lem:3.4}).
\end{proof}
Following the preceding example listed in Table 1, we show the different intervals of $a$ such that $\Gamma$ to be non-discrete when $m=8.$ We will see the corresponding solutions $a\in(c_{n},1)$ by condition (1) and $a\in(d_{n},1)$ by condition (2). However there are no solutions for $a$ when $n\leq6$ or $n\geq 130$ by condition (1) and also no solutions for $a$ when $n\leq3$ by condition (2).
\begin{center}
\begin{table}[h]
\caption{Approximate values of $c_{n}$, $d_{n}.$}
\begin{tabular}{|c|c|c|}
\hline
n&$c_{n}$&$d_{n}$\\
\hline
4&---&0.99961\\
5&---&0.99419\\
6&---&0.99289\\
7&0.99170&0.99279\\
8&0.98685&0.99299\\
9&0.98459&0.99323\\
10&0.98363&0.99346\\
20&0.98750&0.99442\\
30&0.99147&0.99464\\
100&0.99911&0.99480\\
200&---&0.99481\\
\hline
\end{tabular}
\end{table}
\end{center}
\begin{rem}
Non-discrete complex hyperbolic triangle groups of type $(n,n,\infty)$ has been investigated by some authors, one can refer to \cite{Kam3},\cite{KPT1},\cite{KPT2}. Table 1 and Table 2 are extension of their results for different type of complex hyperbolic triangle groups.
\end{rem}
\section{Complex hyperbolic triangle groups of type $(n,\infty,\infty)$}
\par In this section, the aim is to consider the non-discrete classes of $\Gamma$ of type $(n,\infty,\infty)$. For convenience, we rewrite the three normalised polar vectors of $\Gamma$
\begin{displaymath}
{p_{1}} =
\left[ \begin{array}{ccc}
0\\
1 \\
0
\end{array} \right],\quad
{p_{2}} =
\left[ \begin{array}{ccc}
1 \\
-1 \\
1
\end{array} \right],\quad
{p_{3}} =
\left[ \begin{array}{ccc}
1\\
-s e^{-i\theta}\\
s e^{-i\theta}
\end{array} \right],
\end{displaymath}
where $s=\textmd{cos}(\pi/n).$ Then the matrix representation of the three corresponding complex reflections can easily be obtained. In what follows we still assume that $a=\cos \theta$.  A simple computation yields $\tau =\textmd{tr}(I_{1}I_{2}I_{3})= -7+8 e^{i\theta} \cos(\pi/n)-2 \cos(2\pi/n)$ and the discriminant function (\ref{dis})
\\\begin{align*}
f(\tau)=2048 -
& 10240 a s +1792 s^2 + 21760 a^2 s^2 - 16384 a s^3 -16384 a^3 s^3 + 7680 s^4\\
& + 22528 a^2 s^4 - 18944 a s^5+ 3840 s^6 +4096 a^2 s^6 - 2048 a s^7 + 256 s^8.\\
\end{align*}
\par For different n, the interval of $a$ and the value of the parameter angular invariant $\theta$ such that $\Gamma$ to be non-discrete can be solved by Theorem \ref{thm3.1}. We observe that for $n\geq8$, there are solutions ($\alpha_{n},\beta_{n}$) for $a$, which lie between 0 and 1. But otherwise there are no solutions. Later we tabulate this interval of $\cos \theta$ with other two intervals which are defined after Corollary \ref{cro5.1}.
\begin{cor}\label{cro5.1}
If $\Gamma$ of type $(n,\infty,\infty)$ satisfies
\\(1) $7\leq n<\infty$ and $\Big|\cos^2(\frac{\pi}{n})-4 \cos(\frac{\pi}{n})\cos \theta+3\Big|<\frac{1}{2} \sin(\frac{\pi}{n}),$ or
\\(2) $|u-2iv|+4u<\frac{1}{4},$
where $u=\cos^2(\frac{\pi}{n})-2 \cos(\frac{\pi}{n})\cos \theta+1$,
$v=\cos(\frac{\pi}{n})\sin \theta$.
\\Then $\Gamma$ will be non-discrete.
\end{cor}
\par The proof of this theorem is obvious by letting $m$ to be $\infty$ if we see Theorem \ref{thm3.2}. Thence from the assumption $a=\cos\theta$ and $s=\cos(\pi/n)$, we know if
$$35 - 96 a s + 25 s^2 + 64 a^2 s^2 - 32 a s^3 + 4 s^4<0,\quad \textmd{or}$$
$$\sqrt{1 - 4 a s + 6 s^2 - 4 a s^3 + s^4}<\frac{-15 + 32 a s - 16 s^2}{4},$$
then there are intervals $(\gamma_{n},1)$, $(\eta_{n},1)$ of $a$ for $\Gamma$ to be non-discrete. The following Table 3 shows the intervals of $\cos$ $\theta$.

\par Let $\Gamma$ be a complex hyperbolic triangle group of type $(n,\infty,\infty;k),$ where $k$ is the order of $I_3I_1I_3I_2$. By simple computation, we have
$$\textmd{tr}(I_3I_1I_3I_2)=3+16s^2-16sa.$$
Denote $\textmd{tr}(I_3I_1I_3I_2)$ by $t$, then $f(t)=16384(a-s)^3s^3(-1+4(a-s)s).$ Therefore $I_3I_1I_3I_2$ will be an regular elliptic when $a\in\big(s,\frac{1+4s^2}{4s}\big)$. Especially $a=s$ leads $I_3I_1I_3I_2$ to be unipotent parabolic, while if $a=\frac{1+4s^2}{4s}$ then $I_3I_1I_3I_2$ will be a boundary elliptic. In the following we will give two related examples.

\begin{eg}
Discreteness of $\Gamma_{n}$ of type $(n,\infty,\infty)$ whose angular invariant $\alpha=\frac{\pi}{n}\textmd{ (i.e. } a=s ).$
\end{eg}
\par By computing $\tau=\textmd{tr}(I_{1}I_{2}I_{3})=-3 + 2 \cos\big(2 \pi/n\big) + 4 i \sin\big(2 \pi/n\big),$ we have
$$f(\tau)=128\big(7-9\cos(2\pi/n)\big)\big(\sin(\pi/n)\big)^6.$$
 $I_{1}I_{2}I_{3}$ will be a regular elliptic element when $n\geq10$ which leads $\Gamma_{n}$ to be non-discrete.
Additionally the inequality stated in Corollary \ref{cro5.1} (1) is equivalent to
$$\textmd{sin}\Big(\frac{\pi}{n}\Big)<\frac{1}{6},$$
i.e. $n\geq19$. Meanwhile the condition (2) yields
$$\sqrt{16+32a^2-48a^4}<-15+16a^2,$$
i.e. $n\geq 61.$ Therefore $\Gamma_{n}$ will be non-discrete when $n\geq10.$
\par Specifically, when $n=4$, See \text{\cite[Theorem 2.1]{KPT1}}, we will see
\begin{displaymath}
{I_{1}} =
\left[ \begin{array}{ccc}
-1 &\quad 0 \quad& 0 \\
0 & 1 & 0 \\
0 & 0 & -1
\end{array} \right],\quad
{I_2} =
\left[ \begin{array}{ccc}
1 &\quad -2 \quad& -2 \\
-2 & 1 & 2 \\
2 & -2 & -3
\end{array} \right],
\end{displaymath}
\begin{displaymath}
{I_3} =
\left[ \begin{array}{ccc}
1 &\quad -1-i  \quad&  -1-i \\
-1+i & 0 & 1 \\
1-i & -1 & -2
\end{array} \right].
\end{displaymath}
Obviously all of the matrix entries are in $\mathbb{Z}[i]$ which is a discrete subring of $\mathbb{C}.$ Therefore $\Gamma_{4}$ is discrete. Here $I_{1}I_{2}I_{3}$ is a loxodromic element.

\begin{center}
\begin{table}[htb]
\caption{Approximate values of $\alpha_{n},$ $\beta_{n},$ $\gamma_{n}$, $\eta_{n}.$}
\begin{tabular}{|c|c|c|c|c|}
\hline
n&$\alpha_{n}$&$\beta_{n}$&$\gamma_{n}$&$\eta_{n}$\\
\hline
4&---&---&---&0.99959\\
5&---&---&---&0.99857\\
6&---&---&---&0.99624\\
7 & ---& --- & 0.99748&0.99524\\
8 & 0.93724&0.93784&0.99099&0.99482\\
9&0.94201&0.94794&0.98756&0.99463\\
10&0.94476&0.95631&0.98575&0.99454\\
15&0.94993&0.97914&0.98472&0.99451\\
20&0.95142&0.98799&0.98647&0.99455\\
40&0.95272&0.99694&0.99171&0.99461\\
100&0.95306&0.99951&0.99632&0.99463\\
200&0.95311&0.99988&0.99809&0.99464\\
\hline
\end{tabular}
\end{table}
\end{center}

\begin{eg}
$(7,\infty,\infty;5)$ is non-discrete.
\end{eg}
\par From the assumption about the trace of $I_{3132},$ we can deduce $3+16s^2-16sa=1+2 \cos(2\pi/k)$, i.e.
$$\cos(\theta)=\frac{8s^2-\cos(2\pi/k)+1}{8s},$$
where $s=\cos(\pi/7).$ It follows from Table 3 that $\Gamma$ is non-discrete when $0.28621\leq\cos(2\pi/k)\leq0.32052.$ Then it is easy to see that $\Gamma$ of type $(7,\infty,\infty;5)$ is non-discrete.
\\
\\\textbf{Acknowledgement} I am grateful to Professor Toshiyuki Sugawa for his patient guidance and valuable suggestions. I would like to thank Professor John Parker for his comments to this paper and consistent help. I also thank Professor Xiantao Wang for assistance with the earlier version of the manuscript.

\end{document}